\documentclass[11pt]{amsart} %[10pt]{article}
\usepackage{amsmath, amssymb, amsthm}
\usepackage{amscd}
\usepackage[dvips]{graphicx}
\usepackage{bm}

\usepackage[all]{xy}
\usepackage[pdftitle={Extremal norms},pdfauthor={Fili, Miner},colorlinks=false,pdfstartview={}]{hyperref}

\newtheorem{thm}{Theorem}
\newtheorem*{thm1}{Theorem \ref{thm:extremal-infimum}}

\newtheorem*{thm4}{Theorem \ref{thm:minimizing-mod-wk}}
\newtheorem{conj}{Conjecture}
\newtheorem{prop}{Proposition}
\newtheorem{lemma}[prop]{Lemma}

\newtheorem{cor}[prop]{Corollary}

\newtheorem*{thm*}{Theorem}
\newtheorem*{alg*}{Algorithm}
\newtheorem*{lemma*}{Lemma}

\theoremstyle{remark}
\newtheorem{rmk}[prop]{Remark}
\newtheorem*{rmk*}{Remark}
\newtheorem{notation}[prop]{Notation}
\newtheorem*{notation*}{Notation}

\theoremstyle{definition}
\newtheorem{defn}[prop]{Definition}

\numberwithin{equation}{section}
\numberwithin{prop}{section}

\newcommand{\mybf}{\mathbb}

\newcommand{\bR}{\mybf{R}}

\newcommand{\bN}{\mybf{N}}

\newcommand{\bQ}{\mybf{Q}}

\newcommand{\cA}{\mathcal{A}}
\newcommand{\cP}{\mathcal{P}}

\newcommand{\cL}{\mathcal{L}}

\newcommand{\cG}{\mathcal{G}}
\newcommand{\cK}{\mathcal{K}}

\newcommand{\al}{\alpha}
\newcommand{\bmath}[1]{\bm{#1}}
\newcommand{\ovrln}{\bmath}

\newcommand{\Gal}{\operatorname{Gal}}

\newcommand{\supp}{\operatorname{supp}}

\newcommand{\ra}{\rightarrow}

\newcommand{\lcm}{\operatorname{lcm}}

\newcommand{\ep}{\epsilon}

\newcommand{\lfrac}[2]{\left(\frac{#1}{#2}\right)}

\newcommand{\Tor}{\operatorname{Tor}}

\newcommand{\Stab}{\operatorname{Stab}}
\newcommand{\ord}{\operatorname{ord}}

\newcommand{\Qbar}{\overline{\mybf{Q}}}

\def\talltareesidedbox#1{\setbox0=\hbox{$#1$}\dimen0=\wd0 \advance\dimen0 by3pt\rlap{\hbox{\vrule height10pt width.4pt
 depth2pt \kern-.4pt\vrule height10.4pt width\dimen0 depth-10pt\kern-.4pt \vrule height10pt width.4pt depth2pt}}
 \relax \hbox to\dimen0{\hss$#1$\hss}}%s\ignorespaces}
\def\tareesidedbox#1{\setbox0=\hbox{$#1$}\dimen0=\wd0 \advance\dimen0 by3pt\rlap{\hbox{\vrule height8pt width.4pt
 depth2pt \kern-.4pt\vrule height8.4pt width\dimen0 depth-8pt\kern-.4pt \vrule height8pt width.4pt depth2pt}}
\relax \hbox to\dimen0{\hss$#1$\hss}}%s\ignorespaces}
\def\shorttareesidedbox#1{\setbox0=\hbox{$#1$}\dimen0=\wd0 \advance\dimen0 by3pt\rlap{\hbox{\vrule height7pt width.4pt
 depth2pt \kern-.4pt\vrule height7.4pt width\dimen0 depth-7pt\kern-.4pt \vrule height7pt width.4pt depth2pt}}
 \relax \hbox to\dimen0{\hss$#1$\hss}}%s\ignorespaces}
\newcommand{\house}[1]{\tareesidedbox{#1}}

\newcommand{\mhat}{\widehat{m}}

\newcommand{\bal}{\bmath{\alpha}}

\newcommand{\bbeta}{\bmath{\beta}}
\newcommand{\btau}{\bmath{\tau}}
\newcommand{\bgamma}{\bmath{\gamma}}
\newcommand{\mbr}{m}

\newcommand{\Norm}{\operatorname{Norm}}

\title[Extremal norms]{Norms extremal with respect to the Mahler measure}\author[Fili]{Paul Fili}
\address{Department of Mathematics\\ University of Texas at Austin, TX 78712}
\email{pfili@math.utexas.edu}
\author[Miner]{Zachary Miner}
\address{Department of Mathematics\\ University of Texas at Austin, TX 78712}
\email{zminer@math.utexas.edu}
\date{\today}

\begin{document}

\begin{abstract}
In a previous paper, the authors introduced several vector space norms on the space of algebraic numbers modulo torsion which corresponded to the Mahler measure on a certain class of numbers and allowed the authors to formulate $L^p$ Lehmer conjectures which were equivalent to their classical counterparts. In this paper, we introduce and study several analogous norms which are constructed in order to satisfy an extremal property with respect to the Mahler measure. These norms are a natural generalization of the metric Mahler measure introduced by Dubickas and Smyth. We evaluate these norms on certain classes of algebraic numbers and prove that the infimum in the construction is achieved in a certain finite dimensional space.
\end{abstract}

\maketitle

\tableofcontents

\section{Introduction}
\subsection{Background}
Let $K$ be a number field with set of places $M_K$. For each $v\in M_K$ lying over a rational prime $p$, let $\|\cdot\|_v$ be the absolute value on $K$ extending the usual $p$-adic absolute value on $\bQ$ if $v$ is finite or the usual archimedean absolute value if $v$ is infinite.
Then for $\al\in K^\times$, the absolute logarithmic Weil height $h$ is given by 
\[
 h(\al) = \sum_{v\in M_K} \frac{[K_v:\bQ_v]}{[K:\bQ]} \log^+ \|\al\|_v 
\]
where $\log^+ t = \max \{\log t,0\}$. As the right hand side above does not depend on the choice of field $K$ containing $\al$, $h$ is a well-defined function mapping  $\Qbar^\times\ra[0,\infty)$ which vanishes precisely on the roots of unity $\Tor(\Qbar^\times)$.  Related to the Weil height is the logarithmic \emph{Mahler measure}, given by 
\[
 m(\al)=(\deg \al)\cdot h(\al),
\]
where $\deg\al=[\bQ(\al):\bQ].$  Perhaps the most important open question regarding the Mahler measure is \emph{Lehmer's conjecture} that there exists an absolute constant $c$ such that
\begin{equation}
 m(\al)\geq c>0\quad\text{for all}\quad \al\in\Qbar^\times\setminus \Tor(\Qbar^\times).
\end{equation}
The question of the existence of algebraic numbers with small Mahler measure was first posed in 1933 by D.H. Lehmer \cite{Lehmer}. The current best known lower bound, due to Dobrowolski \cite{Dob}, is of the form
\[
m(\al)\gg \lfrac{\log \log \deg \al}{\log \deg \al}^3\quad\text{for all}\quad \al\in\Qbar^\times\setminus \Tor(\Qbar^\times)
\]
where the implied constant is absolute.

The Weil height $h$ naturally satisfies the conditions of being a metric on the space
\[
 \cG = \Qbar^\times / \Tor(\Qbar^\times)
\]
of algebraic numbers modulo torsion, and in fact, viewing $\cG$ as a vector space over $\bQ$ written multiplicatively (see \cite{AV}), it is easy to see that $h$ is a vector space norm. The study of the Mahler measure on the vector space of algebraic numbers modulo torsion presents several difficulties absent for the Weil height, first of which is that while $m$ also vanishes precisely on $\Tor(\Qbar^\times)$, unlike $h$, it is not well-defined on the quotient space modulo torsion. To get around that difficulty, Dubickas and Smyth \cite{DS} first introduced the \emph{metric Mahler measure}, which gave a well-defined metric on $\cG$ satisfying the additional property of being the largest metric which descends from a function bounded above by the Mahler measure on $\Qbar^\times$. Later, the first author and Samuels \cite{FS,S} defined the \emph{ultrametric Mahler measure} which satisfies the strong triangle inequality and gives a projective height on $\cG$. It is easy to see that the metric and ultrametric Mahler measures each induce the discrete topology on $\cG$ if and only if Lehmer's conjecture is true.

In this paper we will introduce vector space norms on $\cG$ which satisfy an analogous extremal property with respect to the Mahler measure as the metric Mahler measure does. Before presenting our constructions, let us fix our notation. We denote the $L^p$ Weil heights for $1\leq p<\infty$ by
\[
 h_p(\al) = \bigg(\sum_{v\in M_K} \frac{[K_v:\bQ_v]}{[K:\bQ]}\cdot \left|\log \|\al\|_v\right|^p\bigg)^{1/p}\quad\text{for}\quad\al\in K^\times,
\]
noting that the classical Weil height satisfies $2h=h_1$ (see \cite{AV}) and is well-defined, independent of the choice of $K$. For $p=\infty$, we let
\[
 h_\infty(\al) = \sup_{v\in M_K} \left|\log \|\al\|_v\right| \quad\text{for}\quad\al\in K^\times.
\]
Analogously, we define the $L^p$ Mahler measure on $\Qbar$ by $m_p(\al)=(\deg \al)\cdot h_p(\al)$, and note that $m_1=2m$ is twice the usual Mahler measure.

For an algebraic number $\al\in\Qbar^\times$, we let $\bal\in\cG$ denote its equivalence class modulo torsion. Let  $d:\cG\ra\bN$ be given by
\[
 d(\bal) = \min_{\zeta\in\Tor(\Qbar^\times)} \deg \zeta\al,
\]
where $\zeta\al$ ranges over all representatives of the equivalence class $\bal$. The \emph{minimal logarithmic Mahler measure} is defined to be the function $m:\cG\ra[0,\infty)$ given by 
\[
 m(\bal)=d(\bal)h(\bal).
\]
(Recall that $h_p$ is constant on cosets modulo torsion, in particular, $h_p(\bal)=h_p(\al)$ for all $\al\in\Qbar^\times$.) More generally, we define the \emph{minimal logarithmic $L^p$ Mahler measure}
\[
 \mbr_p(\bal)=d(\bal)h_p(\bal).
\]
This function is called minimal because it yields, for any element $\bal\in\cG$, the minimal logarithmic $L^p$ Mahler measure amongst all of the representatives in $\Qbar^\times$ of our $\bal\in\cG$ , that is,
\[
\mbr_p(\bal) = \min_{\zeta\in\Tor(\Qbar^\times)} m_p(\al\zeta).
\]

Let us now recall the construction of the metric Mahler measure $\mhat: \cG\ra [0,\infty)$ of Dubickas and Smyth \cite{DS}, which is a metric on $\cG$ extremal with respect to the Mahler measure. The (logarithmic) metric Mahler measure is defined by
\[
\widehat{m}(\bal) = \inf_{\al=\al_1\cdots\al_n} \sum_{i=1}^n m(\al_i),
\]
where the infimum is taken over all possible ways of writing any representative of $\bal$ as a product of other algebraic numbers. This construction is extremal in the sense that any other function $g : \cG\ra [0,\infty)$ satisfying
\begin{enumerate}
 \item $g(\bal)\leq m(\bal)$ for all $\bal\in\cG$, and 
 \item $g(\bal\bbeta^{-1})\leq g(\bal)+g(\bbeta)$ for all $\bal,\bbeta\in\cG$,
\end{enumerate}
is then smaller than $\mhat$, that is, $g(\bal)\leq \mhat(\bal)$ for all $\bal\in\cG$. Equivalently, lifting to $\Qbar^\times$ in the natural way, it is easy to see that $\mhat$ satisfies the same extremal property with respect to the logarithmic Mahler measure. This extremal property is characteristic of the metric construction for height functions \cite{DS, DS2, FS}.

\subsection{Main results}
The space $\cG$ has a vector space structure over $\bQ$ (written multiplicatively), so we might ask if there exists a vector space norm satisfying the same extremal property with respect to the Mahler measure. We define the \emph{extremal norm} $\widetilde{m}_p$ associated to $m_p$ to be:
\[
 \widetilde{m}_p(\bal)=\inf_{\bal=\bal_1^{r_1}\cdots\bal_n^{r_n}} \sum_{i=1}^n |r_i|m_p(\bal_i),
\]
where the infimum is taken over all ways of writing $\bal$ as a linear combination of vectors $\bal_i\in\cG$ with $r_i\in\bQ$. We prove that $\widetilde m_p$ is a well-defined vector space norm on $\cG$ which is extremal amongst all seminorms with respect to the Mahler measure, in the sense that if $g : \cG\ra [0,\infty)$ is a function satisfying
\begin{enumerate}
 \item $g(\bal)\leq m_p(\bal)$ for all $\bal\in\cG$,
 \item $g(\bal\bbeta^{-1})\leq g(\bal)+g(\bbeta)$, and
 \item $g(\bal^r)= |r| g(\bal)$ for all $\bal\in\cG,r\in\bQ$,
\end{enumerate}
then $g\leq \widetilde m_p$, that is, $g(\bal)\leq \widetilde m_p(\bal)$ for all $\bal\in\cG$.

Our main result is a finiteness theorem for the extremal norm $\widetilde m_1$ analogous to the main result of \cite{S} for the infimum of the metric Mahler measure. Let $K$ be a number field and let
\[
 V_K = \{\bal^r:r\in\bQ\text{ and }\bal\in K^\times/\Tor(K^\times) \}
\]
be the vector subspace inside $\cG$ spanned by elements of $K^\times/\Tor(K^\times)$. Notice that $\bal\in V_K$ if and only if for any coset representative $\al\in\Qbar^\times$ we have $\al^n\in K^\times$ for some $n\in\bN$. Let $M_K$ be the set of places of $K$, and let $S\subset M_K$ be a finite set of places of $K$, including all archimedean places. Then for any field extension $L/K$, define
\[
 V_{L,S} = \{ \bal \in V_L : \|\bal\|_w=1\text{ for }w|v\in M_K\setminus S\}.
\]
Observe that by Dirichlet's Theorem, $V_{L,S}$ is a finite dimensional vector space.
We then prove the following result: 
\begin{thm}\label{thm:extremal-infimum}
 Let $\bal\in V_K$, where $K$ is Galois. Then there exists a finite set of rational primes $S$, containing the archimedean place, such that
 \[
  \widetilde m_1(\bal) = \sum_{F\subseteq K} [F:\bQ]\cdot h_1(\bal_F)
 \]
 where $\bal_F\in \overline{V_{F,S}}$, $\bal = \prod_{F\subseteq K} \bal_F$, and for each pair of fields $E\subset F\subseteq K$,
 \[
  h_1(\bal_F) = \inf_{\bbeta\in V_{E,S}} h_1(\bal_F / \bbeta).
 \]
\end{thm}
\noindent In other words, the norm of each $\bal_F$ is equal to the quotient norm of $\bal_F$ with respect to any subfield.
  
In order to prove our results, we first prove several results related to heights of algebraic numbers modulo multiplicative group actions very much related to the results of \cite{MF}, which we interpret as results about quotient norms. We then construct an $S$-unit projection which allows us to reduce to finite dimensions and conclude with a new theorem that is used to describe the infimum of $\widetilde m_1$:
\begin{thm4}
 For a given $\bal\in V_{L,S}$, there exists $\ovrln{\eta}\in \overline{V_{K,S}}$ such that the following conditions hold:
\begin{enumerate}
 \item  $\displaystyle h_1(\bal\ovrln{\eta}^{-1}) = \inf_{\bbeta\in V_{K,S}} h_1(\bal\bbeta^{-1})$, and
 \item  $\displaystyle h_1(\ovrln{\eta}) + [L:K]h_1(\bal\ovrln{\eta}^{-1}) = \inf_{\bbeta\in V_{K,S}} \big(h_1(\bbeta) + [L:K]h_1(\bal\bbeta^{-1})\big)$. 
\end{enumerate}
\end{thm4}

\subsection{Applications to Lehmer's problem}
Given that the norms $\widetilde m_p$ are extremal with respect to the Mahler measure, it is natural to ask what applications these norms have to the Lehmer problem. Define $\cA\subset\cG$ to be the set of $\boldsymbol 1\neq \bal\in\cG$ which have a representative $\al$ satisfying the following properties:
\begin{enumerate}
 \item $\al$ is an algebraic unit. 
 \item $[\bQ(\al^n):\bQ] = [\bQ(\al):\bQ]$ for all $n\in\bN$.
 \item For any proper subfield $F$ of $K=\bQ(\al)$, $\Norm^K_F(\al)\in \Tor(F^\times)$.
\end{enumerate}
The conditions of the set $\cA$ are exactly, in the terminology of \cite{FM}, that $\bal$ be a unit, Lehmer irreducible, and projection irreducible, respectively. Then in \cite[Theorem 4]{FM} it is proven that for any $1\leq p\leq\infty$ there exists a constant $c_p$ such that
\begin{equation}\label{eqn:Lp-lehmer-conj}
  m_p(\al)=(\deg\al)\cdot h_p(\al)\geq c_p>0\quad\text{for all}\quad \al\in\Qbar^\times\setminus \Tor(\Qbar^\times)
\end{equation}
if and only if
\begin{equation}
 m_p(\bal)=d(\bal)\cdot h_p(\bal)\geq c_p>0\quad\text{for all}\quad \bal\in\cA.
\end{equation}
We note that equation \eqref{eqn:Lp-lehmer-conj} is equivalent to the Lehmer conjecture for $p=1$ and the Schinzel-Zassenhaus conjecture for $p=\infty$ \cite[Proposition 4.1]{FM}. Therefore we formulate the following conjecture:
\begin{conj}\label{conj:bound-dhhat-on-A}
 For each $1\leq p\leq\infty$, there exists a constant $c_p$ such that
\begin{equation}\label{eqn:dhhat-conj}
  \widetilde m_p(\bal)\geq c_p>0\quad\text{for all}\quad \bal\in\cA.
\end{equation}
\end{conj}
\begin{thm}\label{thm:dh-bounds-implies}
 If Conjecture \ref{conj:bound-dhhat-on-A} is true, then \eqref{eqn:Lp-lehmer-conj} holds.
\end{thm}
In particular, for $p=1$ \eqref{eqn:dhhat-conj} implies that Lehmer's conjecture is true, and for $p=\infty$ equation \eqref{eqn:dhhat-conj} implies that the Schinzel-Zassenhaus conjecture is true.

For $p\neq 2$, we are unable to prove the converse to Theorem \ref{thm:dh-bounds-implies}. However, when $p=2$ we are able to prove that:
\begin{thm}\label{thm:L2-Lehmer-equiv}
There exists a constant $c_2$ such that  
 \begin{equation*}
  m_2(\al)=(\deg\al)\cdot h_2(\al)\geq c_2>0\quad\text{for all}\quad \al\in\Qbar^\times\setminus \Tor(\Qbar^\times)
 \end{equation*}
if and only if
 \[
  \widetilde m_2 (\bal)\geq c_2>0\quad\text{for all}\quad \bal\in\cA.
 \]
\end{thm}
\begin{proof}
 In \cite {FM}, we construct a norm $\|\cdot\|_{m,2}$, and prove in \cite[Theorem 4]{FM} that bounding $\|\cdot\|_{m,2}$ away from zero on $\cA$ is equivalent to bounding $m_2$ away from zero on $\Qbar^\times\setminus \Tor(\Qbar^\times)$.  Further, in \cite[Theorem 6 et seq.]{FM} we prove that $\|\bal\|_{m,2}\leq m_2(\bal)$ for all $\bal\in\cG$.  It follows by the extremal property for $\widetilde m_2$ that 
 \[
  \|\bal\|_{m,2}\leq \widetilde m_2(\bal)\leq m_2(\bal)
 \]
 for all $\bal\in\cG$, and the claim now follows.
\end{proof}

\thanks{At this point the authors would like to acknowledge Jeffrey Vaaler for many helpful conversations in general, and specifically for his contributions to Lemma \ref{lemma:alpha-v-exists}, as well as Clayton Petsche and Felipe Voloch for helpful remarks regarding this same lemma.}

 The format of this paper is as follows. In Section \ref{sect:Proj-delta-lemmas} we prove basic results about degree functions on $\cG$ and projections onto subspaces.  In Section \ref{sect:construction} we will construct the extremal norms $\widetilde m_p$ arrived at by the infimum process and examine explicit classes of algebraic numbers for which we can compute the value of the norms (for example, on surds and in the $p=1$ case on Salem and Pisot numbers). Lastly in Section \ref{sect:dh1-infimum} we will study $\widetilde m_1$ in particular and prove for any given class of an algebraic number that the infimum in its construction is attained in a finite dimensional vector space.

\section{Preliminary Lemmas}\label{sect:Proj-delta-lemmas}
\subsection{Subspaces associated to number fields}
We will now prove some lemmas regarding the relationship between certain subspaces determined by number fields. Let $G=\Gal(\Qbar/\bQ)$ and let us define
\[
 \cK = \{ K/\bQ : [K:\bQ]<\infty\}\quad\text{and}\quad 
 \cK^G = \{ K\in\cK : \sigma K =K\ \forall \sigma\in G\}.
\]
Let us briefly recall the combinatorial properties of the sets $\cK$ and $\cK^G$ partially ordered by inclusion. Recall that $\cK$ and $\cK^G$ are \emph{lattices}, that is, partially ordered sets for which any two elements have a unique greatest lower bound, called the \emph{meet}, and a least upper bound, called the \emph{join}. Specifically, for any two fields $K,L$, the meet $K\wedge L$ is given by $K\cap L$ and the join $K\vee L$ is given by $KL$. If $K,L$ are Galois then both the meet (the intersection) and the join (the compositum) are Galois as well, thus $\cK^G$ is also a lattice. Both lattices have a minimal element, namely $\bQ$, and are \emph{locally finite}, that is, between any two fixed elements we have a finite number of intermediate elements.

For each $K\in\cK$, let
\[
 V_K = \{\bal^r:r\in\bQ\text{ and }\bal\in K^\times/\Tor(K^\times) \}.
\]
Then $V_K$ is the subspace of $\cG$ spanned by elements of $K^\times/\Tor(K^\times)$.  We call a subspace of the form $V_K$ for $K\in\cK$ a \emph{distinguished} subspace.  Suppose we fix an algebraic number $\bal\in\cG$. Then the set 
\[
\{ K\in\cK : \bal\in V_K\}
\]
forms a sublattice of $\cK$, and by the finiteness properties of $\cK$ this set must contain a unique minimal element.
\begin{defn}
 For any $\bal\in \cG$, the \emph{minimal field} is defined to be the minimal element of the set $\{ K\in\cK : \bal\in V_K\}$. We denote the minimal field of $\bal$ by $K_{\bal}$.
\end{defn}

Note that the action of $G=\Gal(\Qbar/\bQ)$ on $\cG$ is well-defined (see \cite{AV}).

\begin{lemma}\label{lemma:Stab-min-field}
 For any $\bal\in\cG$, we have $\Stab_{G}(\bal) = \Gal(\Qbar/K_{\bal}) \leq G$.
\end{lemma}
\begin{notation}
 By $\Stab_G(\bal)$ we mean the $\sigma\in G$ such that $\sigma \bal = \bal$. As this tacit identification is convenient we shall use it throughout with no further comment.
\end{notation}
\begin{proof}
Clearly $\Gal(\Qbar/K_{\bal})\leq \Stab_G(\bal)$, as $\al^\ell\in K_{\bal}$ for some $\ell\in\bN$ by definition of $V_{K_{\bal}}$. To see the reverse containment, observe that $K_{\bal}=\bQ(\al^\ell)$ for some $\ell\in\bN$.  Now, for $\sigma\in\Stab_G(\bal)$, we have $\sigma\al=\zeta\al$ for $\zeta\in\Tor(\bQ^\times)$.  Then if $\al^\ell\neq(\zeta\al)^\ell$, there would be a proper subfield of $K_{\bal}$ which contains a power of $\al$, contradicting the definition of $K_{\bal}$.
\end{proof}

\subsection{Lehmer irreducibility}\label{subsect:d-and-l}
Observe that the action of the absolute Galois group $G=\Gal(\Qbar/\bQ)$ is well-defined on the vector space of algebraic numbers modulo torsion $\cG$, and in fact it is easy to see that each Galois automorphism gives rise to a distinct isometry of $\cG$ in the $h_p$ norm (see \cite[\S 2.1]{FM} for more details). Let us denote the image of the class $\bal$ under $\sigma\in G$ by $\sigma\bal$. In order to associate a notion of degree to a subspace in a meaningful fashion so that we can define our norms associated to the Mahler Measure we define the function $\delta : \cG\ra \bN$ by
\begin{equation}\label{eqn:delta-defn}
 \delta(\bal) = \#\{\sigma\bal : \sigma\in G\}=[G:\Stab_G(\bal)]=[K_{\bal}:\bQ]
\end{equation}
to be the size of the orbit of $\bal$ under the Galois action, with the last equality above following from Lemma \ref{lemma:Stab-min-field}.

Observe that since taking roots or powers does not affect the $\bQ$-vector space span, and in particular the minimal field $K_{\bal}$, the function $\delta$ is invariant under nonzero scaling in $\cG$, that is, $\delta(\bal^r)=\delta(\bal)$ for all $0\neq r\in\bQ$. In order to better understand the relationship between our elements in $\cG$ and their representatives in $\Qbar^\times$, we need to understand when an $\bal\in V_K$ has a representative $\al\in K^\times$ (or is merely a root of an element $\al^n\in K^\times$ for some $n>1$). Naturally, the choice of coset representative modulo torsion affects this, and we would like to avoid such considerations. Therefore we define the function $d:\cG\ra\bN$ by
\begin{equation}\label{eqn:d-defn}
 d(\bal) = \min\{\deg \zeta\al : \al\in\Qbar^\times,\ \zeta\in\Tor(\Qbar^\times)\}.
\end{equation}
In other words, for a given $\bal\in\cG$, which is an equivalence class of an algebraic number modulo torsion, $d(\bal)$ gives us the minimum degree amongst all of the coset representatives in $\Qbar^\times$ modulo the torsion subgroup.

A number $\bal\in\cG$ can then be represented by an algebraic number in $K_{\bal}^\times$ if and only if $d(\bal)=\delta(\bal)$. We therefore make the following definition:
\begin{defn}
 We define the set of \emph{Lehmer irreducible} elements of $\cG$ to be the set
\begin{equation}\label{eqn:LI-defn}
 \cL = \{\bal\in\cG : \delta(\bal) = d(\bal) \}.
\end{equation}
The set $\cL$ consists precisely of the $\bal\in\cG$ that can be represented by some $\al\in\Qbar^\times$ of degree equal to the degree of the minimal field $K_{\bal}$ of $\bal$.
\end{defn}
We recall the terminology from \cite{D} that a number $\al\in\Qbar^\times$ is \emph{torsion-free} if $\al/\sigma\al\not\in\Tor(\Qbar^\times)$ for all distinct Galois conjugates $\sigma\al$. Thus, torsion-free numbers give rise to distinct elements $\sigma\bal\in\cG$ for each distinct Galois conjugate $\sigma\al$ of $\al$ in $\Qbar$.

\begin{lemma}\label{lemma:ell-irred}
 We have the following:
 \begin{enumerate}
  \item For each $\bal\in\cG$, there is a unique minimal exponent $\ell(\bal)\in\bN$ such that $\bal^{\ell(\bal)}\in\cL$.
  \item For any $\al\in\Qbar^\times$, we have $\delta(\bal) | \deg\al$.
  \item $\bal\in\cL$ if and only if it has a representative in $\Qbar^\times$ which is torsion-free. 
 \end{enumerate}
\end{lemma}
\begin{proof}
% Every $\bal\in\cG$ satisfies $\bal\in V_{K_{\bal}}$ for the minimal field $K_{\bal}$. Thus for any representative $\al\in\Qbar^\times$ we have $\al^\ell\in K_{\bal}$ for some $\ell$.  But then, $\deg\al^\ell=[K_{\bal}:\bQ]=\delta(\bal)=\delta(\bal^\ell)$, since $\delta$ is invariant under scaling.
 For $\bal\in\cG$, choose a representative $\al\in\Qbar^\times$ and let 
 \[
  \ell = \lcm\{ \ord(\al/\sigma\al) : \sigma\in G\text{ and }\al/\sigma\al\in\Tor(\Qbar^\times)\}
 \]
 where $\ord(\zeta)$ denotes the order of an element $\zeta\in\Tor(\Qbar^\times)$. Then observe that $\al^\ell$ is torsion-free. Now if a number $\beta\in\Qbar^\times$ is torsion-free, then each distinct conjugate $\sigma\beta$ determines a distinct element in $\cG$, so we have
 \[
  \deg \beta = [G : \Stab_G(\bbeta)] = [K_{\bbeta} : \bQ]=\delta(\bbeta).
 \]
 Thus $\deg \al^\ell = \delta(\bal^\ell)$.  
 This proves existence in the first claim, and the existence of a minimum value follows since $\bN$ is discrete. To prove the second claim, observe that $\bQ(\al^\ell)\subset \bQ(\al)$, so with the choice of $\ell$ as above, we have $\delta(\bal)=[\bQ(\al^\ell):\bQ] | [\bQ(\al):\bQ]=\deg \al$ for all $\al\in\Qbar^\times$. The third now follows immediately.
\end{proof}
It is proven in \cite{FM} that in fact, the minimal value $\ell(\bal)$ satisfies $d(\bal) = \ell(\bal) \delta(\bal)$.

\subsection{Projections to $V_K$}
For $K\in\cK$, we wish to define an operator that projects an element $\bal\in\cG$ onto the subspace $V_K$.  Let $H=\Gal(\Qbar/K)\leq G$, and let $\sigma_1,\hdots,\sigma_k$ be right coset representatives from $\Stab_H(\bal)\leq H$, with $k=[H:\Stab_H(\bal)]$. Define the map $P_K$ on elements of $\cG$ via
\begin{equation}
  P_K(\bal) =\left(\prod_{i=1}^k\sigma_i(\bal)\right)^{1/k}.
\end{equation}

\begin{lemma} \label{lemma:PK-to-VK}
Let $\bal\in\cG$.  Then:
\begin{enumerate}
 \item $P_K(\bal^r)=P_K(\bal)^r.$
 \item $P_K(\bal)\in V_K.$
 \item $P_K(\bal)=\bal$ for all $\bal\in V_K$.
\end{enumerate}
 
\end{lemma}
\begin{proof}

(1) $P_K(\bal^r)=P_K(\bal)^r$ follows from its definition in terms of the Galois action on $\cG$.

(2) By scaling if necessary, we may assume $\bal\in\cL$.  Choose a torsion-free representative $\al\in\Qbar^\times$.  Then, the result will follow from $N_K^{K(\al)}=\prod_{i=1}^k\sigma_i(\al)$, since $N_K^{K(\al)}(\al)\in K^\times$.  To see this, note that for $\al$ torsion-free, $\al/\sigma(\al)$ is never a nontrivial torsion element, so its orbit in $\Qbar^\times$ and its $\cG$ orbit coincide.  

(3) Again, we may assume $\bal\in\cL$, so that a torsion-free representative $\al\in K^\times$.  Then $N_K^{K(\al)}=\al^k$, and $P_K(\bal)=\bal$ follows.
\end{proof}

\begin{prop}\label{prop:PK-continuous-wrt-Weil-norm}
 Let $\bQ\subset K\subset\Qbar$ be an arbitrary field. Then $P_K$ is a projection onto $V_K$ of norm one with respect to the $L^p$ norms for $1\leq p\leq \infty$.
\end{prop}
\begin{proof}
 By Lemma \ref{lemma:PK-to-VK}, it follows that $P_K^2=P_K$ is a projection onto $V_K$.  Then for $1\leq p\leq\infty$, 
\[
h_p(P_K\bal)=h_p \left(\sigma_1(\bal)\cdots \sigma_k(\bal)\right)^{1/k}\leq\frac{1}{k}\sum_{i=1}^kh_p(\sigma_i\bal)=\frac{1}{k}\sum_{i=1}^kh_p(\bal)=h_p(\bal),
\]
since the Weil $p$-height is invariant under the Galois action.  This proves $P_K$ has operator norm $\|P_K\|\leq1$, and since $V_\bQ$ is fixed for every $P_K$, we get $\|P_K\|=1$.
\end{proof}

As a corollary, if we let $\cG_p$ denote the completion of $\cG$ under the Weil $p$-norm $h_p$ and extend $P_K$ by continuity, we obtain:

\begin{cor}
 The subspace $\overline{V_K}\subset\cG_p$ is complemented in $\cG_p$ for all $1\leq p\leq\infty$.
\end{cor}
As $\cG_2=L^2(Y,\lambda)$ is the $L^2$ space for a certain measure space $(Y,\lambda)$ constructed explicitly in \cite{AV}, and thus a Hilbert space, more is in fact true:

\begin{prop}\label{prop:PK-is-orthogonal}
 For each $K\in\cK$, $P_K$ is the orthogonal projection onto the subspace $\overline{V_K}\subset \cG_2$.
\end{prop}
\begin{proof}
 Observe that $P_K$ is idempotent and has operator norm $\|P_K\|=1$ with respect to the $L^2$ norm, and any such projection in a Hilbert space is orthogonal (see \cite[Theorem III.1.3]{Y}).
\end{proof}

We now explore the relationship between the Galois group and the projection operators $P_K$ for $K\in\cK$.

\begin{lemma}\label{lemma:PKCommWithG}
 For any field $K\subseteq\Qbar$ and $\sigma\in G$, 
 \[
  \sigma P_K = P_{\sigma K}\, \sigma.
 \]
 Equivalently, $P_K\, \sigma = \sigma P_{\sigma^{-1}K}$.
\end{lemma}
\begin{proof}
 We prove the first form, the second obviously being equivalent.  Let $H=\Gal(\Qbar/K)$, and note that if $\tau\in H$, then $\sigma\tau\sigma^{-1}\in\Gal(\Qbar/\sigma K)$.  Then by the definition of $P_K$: 
 \begin{align*}
 \sigma P_K \bal & = \sigma\left(\sigma_1\bal\cdots\sigma_k\bal\right)^{1/k}\\
& = \left(\sigma\sigma_1\bal\cdots\sigma\sigma_k\bal\right)^{1/k}\\
& = \left(\sigma\sigma_1(\sigma^{-1}\sigma)\bal\cdots\sigma\sigma_k(\sigma^{-1}\sigma)\bal\right)^{1/k}\\
& = \left((\sigma\sigma_1\sigma^{-1})\sigma\bal\cdots(\sigma\sigma_k\sigma^{-1})\sigma\bal\right)^{1/k}\\
& =P_{\sigma K}(\sigma\bal).\qedhere
\end{align*}
\end{proof}

We will be particularly interested in the case where the projections $P_K,P_L$ commute with each other (and thus $P_K P_L$ is a projection to the intersection of their ranges). To that end, let us determine the intersection of two distinguished subspaces:
\begin{lemma}\label{lemma:intersection-of-VK-VL}
 Let $K,L\subset\Qbar$ be extensions of $\bQ$ of arbitrary degree. Then the intersection $V_K\cap V_L = V_{K\cap L}$.
\end{lemma}
\begin{proof}
 Let $\al^m\in K$ and $\al^n\in L$ for some $m,n\in\bN$. Then $\al^{mn}\in K\cap L$, so $\bal\in V_{K\cap L}$. The reverse inclusion is obvious.
\end{proof}
\begin{lemma}\label{lemma:PKcommute}
 Suppose $K\in\cK$ and $L\in\cK^G$. Then $P_K$ and $P_L$ commute, that is,
 \[
  P_K P_L = P_{K\cap L} = P_L P_K.
 \]
 In particular, the family of operators $\{P_K : K\in\cK^G\}$ is commuting.
\end{lemma}
\begin{proof}
 It suffices to prove $P_K(V_L)\subset V_L$, as this will imply that $P_K(V_L)\subset V_K\cap V_L=V_{K\cap L}$ by the above lemma, and thus that $P_K P_L$ is itself a projection onto $V_{K\cap L}$, implying $P_K P_L=P_{K\cap L}$, and since any orthogonal projection is equal to its adjoint, we find that $P_{K\cap L} = P_L P_K$ as well. To prove that $P_K(V_L)\subset V_L$, observe that for $\bal\in V_L$,
 \[
  P_K(\bal) = \left(\sigma_1\bal\cdots \sigma_k\bal\right)^{1/k}
 \]
 where the $\sigma_i$ are right coset representatives of $\Stab_H(\bal)$ in $H=\Gal(\Qbar/K)$. However, $\sigma(V_L)=V_L$ for $\sigma\in G$ since $L$ is Galois, and thus, $P_K(\bal)\in V_L$ as well. But $P_K(\bal)\in V_K$ by construction and the proof is complete.
\end{proof}

From these facts, we derive the following useful lemma:
\begin{lemma}\label{lemma:delta-PK-leq-for-K-in-Kg}
 If $K\in\cK^G$, then $\delta(P_K\bal)\leq \delta(\bal)$ for all $\bal\in\cG$.
\end{lemma}
\begin{proof}
 Let $F= K_{\bal}$. Since $K\in\cK^G$, we have by Lemma \ref{lemma:PKcommute} that $P_K \bal = P_K(P_F \bal) = P_{K\cap F} \bal$. Thus, $P_K \bal \in V_{K\cap F}$, and so $\delta(P_K\bal)\leq [K\cap F:\bQ]\leq [F:\bQ]=\delta(\bal)$.
\end{proof}

\section{The extremal norms}\label{sect:construction}
\subsection{Construction}
The aim of this section is to construct norms extremal with respect to the minimal Mahler measure. Let us begin by recalling the metric construction, as applied in \cite{DS}:
\begin{defn}
For $f:\cG\ra[0,\infty)$, the \emph{metric height associated to $f$} is defined to be the function $\widehat{f}:\cG\ra [0,\infty)$ given by
\[
 \widehat{f}(\bal) = \inf_{\bal=\bal_1\cdots\bal_n} \sum_{i=1}^n f(\bal_i), 
\]
where the infimum ranges over all possible factorizations $\bal=\bal_1 \cdots\bal_n$ in $\cG$.
\end{defn}
\begin{prop}\label{prop:metric-height-prop}
 Suppose $f(\bal^{-1})=f(\bal)$ for all $\bal\in\cG$.  Then the function $\widehat{f}$ satisfies:
\begin{enumerate}
\item $\widehat{f}(\bal)\leq f(\bal)$ for all $\bal\in\cG$.
\item $\widehat{f}(\bal^{-1})=\widehat{f}(\bal)$ for all $\bal\in\cG$.
\item $\widehat{f}(\bal\bbeta^{-1})\leq \widehat{f}(\bal)+\widehat{f}(\bbeta)$ for all $\bal,\bbeta\in\cG$.
\item The zero set $Z(\widehat{f})=\{\bal\in\cG : \widehat{f}(\bal)=0\}$ is a subgroup of $\cG$, and $\widehat{f}$ is a metric on $\cG/Z(\widehat{f})$.
\end{enumerate}
It is the largest function that does so, that is, for any other function $g$ which satisfies the above conditions, we have $g(\bal)\leq \widehat{f}(\bal)$ for all $\bal\in\cG$. In particular, if $f$ already satisfies the triangle inequality, then $\widehat{f}=f$.
\end{prop}

This last property of being the largest metric less than or equal to $f$ we call the \emph{extremal property}. The construction of metric heights only uses the group structure of $\cG$, and ignores the vector space structure. If we wish to respect scaling in $\cG$ as well, we arrive at the notion of a norm height:
\begin{defn}\label{defn:norm-height}
Let $f:\cG\ra[0,\infty)$ be a given function. We define the \emph{norm height associated to $f$} to be the function $\widetilde{f}:\cG\ra [0,\infty)$ given by
\[
 \widetilde{f}(\bal) = \inf_{\bal=\bal_1^{r_1}\cdots\bal_n^{r_n}} \sum_{i=1}^n |r_i|\,f(\bal_i),
\]
where the infimum ranges over all possible factorizations $\bal=\bal_1^{r_1}\cdots\bal_n^{r_n}$ in $\cG$ with $\bal_i\in\cG$ and $r_i\in\bQ$.
\end{defn}
\begin{prop}\label{prop:norm-height-properties}
The norm height $\widetilde{f}$ satisfies the properties:
\begin{enumerate}
 \item $\widetilde{f}(\bal)\leq f(\bal)$ for all $\bal\in\cG$.
 \item $\widetilde{f}(\bal^r) = |r|\widetilde f(\bal)$ for all $r\in\bQ$ and $\bal\in\cG$.
 \item $\widetilde{f}(\bal\bbeta^{-1}) \leq \widetilde{f}(\bal)+\widetilde{f}(\bbeta)$ for all $\bal,\bbeta\in\cG$.
 \item The zero set $Z(\widetilde{f})=\{\bal\in\cG : \widetilde{f}(\bal)=0\}$ is a vector subspace of $\cG$.
\end{enumerate}
Thus, $\widetilde{f}$ is a seminorm on $\cG$, and a norm on $\cG/Z(\widetilde{f})$. It is the largest function on $\cG$ that satisfies the above properties, that is, for any other function $g$ which satisfies the above conditions, we have $g(\bal)\leq \widetilde{f}(\bal)$ for all $\bal\in\cG$. In particular, if $f$ is already a seminorm on $\cG$, then $\widetilde{f}=f$. 
\end{prop}
The proof of Proposition \ref{prop:norm-height-properties} follows easily from the definitions.  As above, we refer to the last part of the proposition as the \emph{extremal property} of the norm height construction. Observe that if $f$ satisfies the scaling property $f(\bal^r) = |r|f(\bal)$, then $\widetilde f=\widehat f$ and the construction is the same. 
\begin{prop}
 $\widetilde m_p$ is a vector space norm on $\cG$.
\end{prop}
\begin{proof}
It only remains to show that the $\widetilde m_p$ vanishes precisely on the zero subspace of the vector space $\cG$, which is $\{\boldsymbol{1}\}$. Observe that $h_p \leq m_p$, and therefore, by the extremal property,
\[
 h_p(\bal) \leq \widetilde m_p(\bal)\quad\text{for all}\quad\bal\in\cG.
\]
In particular, we see that $\widetilde m_p(\bal)=0$ if and only if $h_p(\bal)=0$, which occurs precisely when $\bal=\boldsymbol{1}$.
\end{proof}
The following result for $\widetilde m_p$ is very useful and will be tacitly used several times in our proofs below:
\begin{prop}
 The norm $\widetilde m_p$ extremal with respect to the Mahler measure $\mbr_p$ is precisely $\widehat{\delta h_p}$, that is, $\widetilde m_p(\bal)=\widehat{\delta h_p}(\bal)$ for all $\bal\in\cG$.
\end{prop}
\begin{proof}
By Lemma \ref{lemma:ell-irred}, there is a unique minimal $\ell\in\bN$ such that
 \[
 d(\bal^\ell) = \delta(\bal).
 \]
  Then it is easy to see that for $\bal\in\cG$, the expression
 \[
  |s/r|\,\mbr_p(\bal^{r/s}) = |s/r|\,d(\bal^{r/s}) h_p(\bal^{r/s})
 \]
 is minimized for $r/s = \ell$, and for that value,
 \[
  |1/\ell|\,d(\bal^{\ell}) h_p(\bal^{\ell}) = \delta(\bal) h_p(\bal).
 \]
 We may then conclude that
 \[
  \widetilde m_p(\bal) = \inf_{\bal=\bal_1^{r_1}\cdots\bal_n^{r_n}} \sum_{i=1}^n |r_i|\,\mbr_p(\bal_i)
  = \inf_{\bal=\bal_1\cdots\bal_n} \sum_{i=1}^n \delta(\bal_i) h_p(\bal_i) = \widehat{\delta h_p}(\bal)
 \]
 which is the desired result.
\end{proof}

\subsection{Explicit values}
We will now compute the values of the norms $\widetilde m_p$ on certain classes of algebraic numbers.

Recall that a \emph{surd} is an algebraic number $\al\in\Qbar^\times$ such that $\al^n\in\bQ^\times$ for some $n\in\bN$. Call $\bal\in\cG$ a surd if one (and therefore all) coset representatives of $\bal$ are surds.
\begin{prop}
 If $\bal\in\cG$ is a surd, then $\widetilde m_p(\bal) = h_p(\bal)$.
\end{prop}
\begin{proof}
 Observe that $\delta(\bal)=1$ for any surd. Since $h_p\leq m_p$ is a norm, we have by the extremal property of $\widetilde m_p$ that
\[
 h_p(\bal) \leq \widetilde m_p(\bal)\leq m_p(\bal)\quad\text{for all}\quad \bal\in\cG.
\]
 But then
\[
 \widetilde m_p(\bal) = \widehat{\delta h_p}(\bal) \leq \delta(\bal)h_p(\bal) = h_p(\bal),
\]
 and therefore we have equality.
\end{proof}

We now consider a class of numbers analogous to the CPS numbers of \cite{DS}.
\begin{lemma}\label{lemma:if-mhat-scales}
 Suppose that $\bal\in\cG$ satisfies $\mhat_p(\bal^n)=n\,\mhat_p(\bal)$ for all $n\in\bN$. Then $\widetilde m_p(\bal) = \mhat_p(\bal)$.
\end{lemma}
\begin{proof}
 Using Theorem \ref{thm:extremal-infimum} we may choose a factorization $\bal = \bal_1\cdots \bal_n$ so that
 \[
  \widehat{\delta h_p}(\bal) = \sum_{i=1}^n \delta(\bal_i)h_p(\bal_i).
 \]
 By Lemma \ref{lemma:ell-irred} we have an exponent $\ell_i = \ell(\bal_i)\in\bN$ such that $\bal_i^{\ell_i}\in\cL$ for $1\leq i\leq n$. Let $k = \lcm\{\ell_1,\ldots,\ell_n\}$. Then observe that
 \[
  k\cdot\widehat{\delta h_p}(\bal) = \sum_{i=1}^n \delta(\bal_i^k)h_p(\bal_i^k)=
  \sum_{i=1}^n d(\bal_i^k) h_p(\bal_i^k)\geq \widehat m_p(\bal^k)= k\,\widehat m_p(\bal).
 \]
 By the extremal property, $\widehat{\delta h_p}(\bal)\leq \widehat m_p(\bal)$, so we must have equality, as claimed.
\end{proof}
\begin{defn}
 Call $\btau\in\cG$ a \emph{Pisot/Salem number} if it has a representative $\tau\in\Qbar^\times$ that can be written as $\tau=\tau_1\cdots\tau_k$ where each $\tau_i>1$ is a Pisot number (that is, an algebraic integer with all of its conjugates strictly inside the unit circle) or a Salem number (an algebraic integer with all of its conjugates on or inside the unit circle, and at least one on the unit circle).
\end{defn}
\begin{prop}
 Every Pisot/Salem number $\btau$ is Lehmer irreducible, that is, $\btau\in\cL$.
\end{prop}
\begin{proof}
 It is easy to see that for a Pisot/Salem number $\btau\in\cG$ and its given representative $\tau>1$ that all other Galois conjugates $\tau'\neq\tau$ have $|\tau'|<|\tau|$. Therefore $\tau$ is torsion-free, since otherwise there would be a conjugate $\tau'=\zeta\tau$ for some $1\neq \zeta\in \Tor(\Qbar^\times)$, having the same modulus as $\tau$, a contradiction. It follows by Lemma \ref{lemma:ell-irred} that $\btau\in\cL$.
\end{proof}
For a Pisot/Salem number $\btau\in\cG$, it is shown in \cite[Theorem 1(c)]{DS} that 
\[
 \mhat_1(\btau^n) = 2 \log \house{\tau^n} \quad\text{for all}\quad n\in\bN.
\]
 Thus, by Lemma \ref{lemma:if-mhat-scales} above, we have the following result:
\begin{prop}
  Let $\btau\in\cG$ be a Pisot/Salem number with given representative $\tau\in\Qbar^\times$. Then
\[ \widetilde m_1(\btau) = \mhat_1(\btau) = 2 \log \house{\tau}.\]
\end{prop}
 Since there exist Pisot and Salem numbers of arbitrarily large degree, and for a Pisot or Salem number $\tau>1$ we have $h_1(\tau)=(2/\deg \tau) \log\house{\tau}$, we easily see that the norms $h_1$ and $\widetilde m_1$ are inequivalent.

\section{The infimum in the $\widetilde m_1$ norm}\label{sect:dh1-infimum}
\subsection{$S$-unit subspaces and quotient norms}
Let $K\in\cK$ be a number field with places $M_K$. Let $S\subset M_K$ be a finite set of places of $K$, including all archimedean places. Then for any finite extension $L/K$, let
\[
 V_{L,S} = \{ \bal \in V_L : \|\bal\|_w=1\text{ for all }w\in M_L\text{ with }w|v\in M_K\setminus S\}.
\]
Then $V_{L,S}$ is the $\bQ$-vector space span inside $V_L$ of the $S'$-units of $L$, where $S'$ is the set of places $w$ of $L$ such that $w|v$ for some $v\in S$. Since we always require that $S$ include the archimedean places, $V_{L,S}$ will always include the vector space span of the units of $L$.

Dirichlet's $S$-unit theorem and, in particular, the non-vanishing of the $S$-regulator, imply the following result:
\begin{prop}\label{prop:WS-is-finite-diml}
 If $S\subset M_K$ as above, then the $\bQ$-vector space $V_{K,S}$ and its completion $\overline{V_{K,S}}$ have finite dimension $\# S-1$. The space $V_{L,S}$ has dimension $\#S'-1$ where $S'$ is the set of places $w$ of $L$ such that $w|v$ for some $v\in S$.
\end{prop}

In what follows below, we will primarily require $S$ to be a set of rational primes, including the infinite prime. Notice that $V_{K,S}\subset V_{L,S}$. One of the goals of this section will be to determine the properties of the quotient norm of $V_{L,S}/V_{K,S}$, in a manner inspired by the initial work of \cite{BM1,BM2} and in particular the more recent work of \cite{MF}.

The main result of this section is the following theorem, which is essentially an analogue for the norm $\widetilde m_1$ of the main result of \cite{S} for the infimum of the metric Mahler measure:
\begin{thm1}
 Let $\bal\in V_K$, where $K$ is Galois. Then there exists a finite set of rational primes $S$, containing the archimedean place, such that
 \[
  \widetilde m_1(\bal) = \sum_{F\subseteq K} [F:\bQ]\cdot h_1(\bal_F)
 \]
 where $\bal_F\in \overline{V_{F,S}}$, $\bal = \prod_{F\subseteq K} \bal_F$, and for each pair of fields $E\subseteq F\subseteq K$,
 \[
  h_1(\bal_F) = \inf_{\bbeta\in V_{E,S}} h_1(\bal_F / \bbeta).
 \]
 In other words, the norm of each $\bal_F$ is equal to the quotient norm of $\bal_F$ with respect to any subfield.
\end{thm1}
In contrast to the main result of \cite{S}, we are unable to prove that this infimum is in fact attained in the vector space of classical algebraic numbers $\cG$, rather than the completion. However, our result is strengthened by the fact that the $S$-unit spaces in which the infimum is attained are finite dimensional real vector spaces. Therefore, if we must pass to the completion, we know that the terms in the infimum are limits of the form $\lim_{n\ra \infty}\bal^{r_n}$ where $\bal\in\cG$ and $r_n$ is a sequence of rational numbers tending to a real limit $r$ as $n\ra\infty$.

In order to prove our results, we must first prove several quotient norm results very much related to the results of \cite{MF}, and then we will construct an $S$-unit projection which will allow us to reduce to the specified situation. Let $S\subset M_K$ be a finite set of places to be specified later, and consider two number fields $K\subset L$. Again let $V_{K,S}$ denote the vector subspace of $V_K$ spanned by the $S$-units of $K$ and let $V_{L,S}$ denote the corresponding subspace of $V_L$. Now for each $v\in S$ let $d_v = [K_v:\bQ_v]$ be the local degree. Rather than following the usual convention and considering the places of $L$ which lie above the places $S$ of $K$, we will consider the $[L:K]$ absolute values which restrict to each place $v$ (that is, we will not consider equivalence on $L$ nor weight such by local degrees). Thus we get $\# S\cdot [L:K]$ absolute values on $L$. Let us fix the $\bal\in V_{L,S}\setminus V_{K,S}$ for which we want to compute the quotient norm modulo $V_{K,S}$. For a given $v\in S$, order the $[L:K]$ absolute values on $L$ which extend $\|\cdot\|_v$ so that
\[ \|\bal\|_{v,1} \leq \|\bal\|_{v,2} \leq \cdots \leq \|\bal\|_{v,[L:K]}.\]
Now we associate to $\bal$ a vector $a\in \bR^{S\times [L:K]}$  via
\begin{align*}
 \varphi : V_{L,S}&\ra \bR^{S\times [L:K]}\\
\bal &\mapsto a = (d_v \log \|\bal\|_{v,i})_{v\in S,\ 1\leq i\leq [L:K]}
\end{align*}
Note that by the product formula and our normalization above, the sum of the components of $a$ is zero. By the ordering above, we also have
\[a_{v,i}\leq a_{v,i+1}\]
for all $v\in S$ and $1\leq i<[L:K]$.
The goal of this section is to prove the following results which will be needed below:
\begin{thm}[de la Maza, Friedman 2008]\label{thm:orig-maza-fried}
 For $\bal\in V_{L,S}$ and the vector $a=\varphi(\bal)\in\bR^{S\times [L:K]}$ with indices ordered as above,
\[
\inf_{\bbeta\in V_{K,S}} h_1(\bal\bbeta^{-1}) = \frac{1}{[L:\bQ]} \sum_{i=1}^{[L:K]}\bigg|\sum_{v\in S} a_{v,i}\bigg|.
\]
Equivalently,
\[
 \|\bal\|_{V_{L,S}/V_{K,S}} = \sum_{i=1}^{[L:K]}\bigg|\sum_{v\in S} a_{v,i}\bigg|,
\]
where we are loosely using the notation $V_{L,S}/V_{K,S}$ for the quotient space of the vector spaces $\overline{\varphi(V_{K,S})}\subset\overline{\varphi(V_{L,S})}\subset \bR^{S\times [L:K]}$ endowed with the $L^1$ norm.
\end{thm}
\begin{rmk}

Note that for a dense open subset of $V_{L,S}$ the components $a_{v,i}$ for a given $v$ may not be assumed distinct, as the distinct places of $L$ lying above $v$ might not be $[L:K]$ in number (for example, if $v$ is a finite place which ramifies or has inertia).  Thus we might always have a certain number of equalities amongst the $\{a_{v,i} : 1\leq i\leq [L:K]\}$ for a given $v$. 
\end{rmk}
We make a slight extension of another result of \cite{MF}:
\begin{thm}\label{thm:maza-fried-mod-units}
Let $\bal\in V_K$ have nonzero support at only the infinite places and one finite place $v$ of $K$. Let $W$ denote the subspace of $V_K$ spanned by the units of $K$. Then there exists $\bbeta\in W$ such that
\[
 h_1(\bal\bbeta^{-1}) = \inf_{\bgamma\in W} h_1(\bal\bgamma^{-1}) 
 =  \frac{1}{[K:\bQ]}\bigg(|d_v\log \|\bal\|_v| + \bigg|\sum_{w|\infty} d_w\log \|\bal\|_w\bigg|\bigg).
\]
\end{thm}

Finally we conclude with a new theorem that will be used to describe the infimum of $\widetilde m_1$:
\begin{thm}\label{thm:minimizing-mod-wk}
 For a given $\bal\in V_{L,S}$, there exists $\ovrln{\eta}\in \overline{V_{K,S}}$ such that the following conditions all hold:
\begin{enumerate}
 \item  $\displaystyle h_1(\bal\ovrln{\eta}^{-1}) = \inf_{\bbeta\in V_{K,S}} h_1(\bal\bbeta^{-1})$, and
 \item  $\displaystyle h_1(\ovrln{\eta}) + [L:K]h_1(\bal\ovrln{\eta}^{-1}) = \inf_{\bbeta\in V_{K,S}} \big(h_1(\bbeta) + [L:K]h_1(\bal\bbeta^{-1})\big)$. 
\end{enumerate}
\end{thm}

We now provide the proofs for the above results.

\begin{proof}[Proof of Theorem \ref{thm:orig-maza-fried}]
 Our proof largely follows that of \cite{MF}, the only addition to the proof being our treatment of the indices $v$ where $a_{v,k} = a_{v,k+1}$ (or in the notation of \cite{MF}, $a_{\sigma_{\tau,k}} = a_{\sigma_{\tau,k+1}}$). Let $a=\varphi(\bal)$ be as above. Notice that 
\[\sum_{v\in S} a_{v,1} \leq \sum_{v\in S} a_{v,2}\leq \cdots\leq \sum_{v\in S} a_{v,[L:K]}.\]
Let $k$ be an index such that 
\[\sum_{v\in S} a_{v,k}\leq 0 \leq \sum_{v\in S} a_{v,k+1}\]
where we let $k=0$ or $k=[L:K]$ if $\sum_{v\in S} a_{v,1}\geq 0$ or $\sum_{v\in S} a_{v,[L:K]}\leq 0$, respectively.
We will assume for the moment that $1\leq k<[L:K]$ and defer the proof for the extreme cases for the moment.
Let $X$ denote the set of $x\in\overline{\varphi(V_{K,S})}\subset \bR^{S\times [L:K]}$ which satisfy the conditions:
\[ 
 a_{v,k} \leq x_v \leq a_{v,k+1}\quad\text{for all}\quad v\in S
\]
and
\[
 \sum_{v\in S} x_v = 0,
\]
where we use $x_v$ to denote the common value of $x_{v,i}$, which must be equal for all $i$ since $x$ arises from $V_{K,S}$. It is easy to see that $X$ is nonempty as it contains, for example,
\[
x_v = a_{v,k} + \frac{-s_k}{s_{k+1} - s_k} (a_{v,k+1}-a_{v,k})
\]
where $s_i = \sum_{v\in S} a_{v,i}$. Notice that 
\begin{multline*}
\|a-x\|_1 = \sum_{i=1}^{[L:K]}\sum_{v\in S} |a_{v,i}-x_v| 
 = \sum_{v\in S} \bigg( \sum_{i=k+1}^{[L:K]}(a_{v,i}-x_v) - \sum_{i=1}^{k}(a_{v,i}-x_v)\bigg) \\ 
 = \sum_{i=1}^{[L:K]}\bigg|\sum_{v\in S} a_{v,i}\bigg| - ([L:K]-2k)\sum_{v\in S}x_v = \sum_{i=1}^{[L:K]}\bigg|\sum_{v\in S} a_{v,i}\bigg|.
\end{multline*}
Since $x=\varphi(\ovrln{\eta})$ for some $\ovrln{\eta}\in \overline{V_{K,S}}$ and $[L:\bQ]\,h_1(\bal\ovrln{\eta}^{-1}) = \|a-x\|_1$, the result will be proven if we can show that the above value is minimal for the function $F : \bR^S \ra \bR$ given by $y\mapsto \|a-y\|_1$ where we again view $y$ as a vector in $\bR^{S\times [L:K]}$ via $y_{v,i} = y_v$.

The function $F$ is clearly convex. Define for each $v$ the standard basis vector $e_v\in\bR^S$ via $(e_v)_v = 1$ and $(e_v)_w = 0$ for $w\neq v\in S$. For each $v$ with $a_{v,k}=a_{v,k+1}$, observe that if $y\in\bR^S$ satisfies $y_v=a_{v,k}$ for this particular $v$, then
\[F(y+ t\, e_v) = F(y) + |t|\quad\text{where}\quad t\in\bR.\]
Thus the function $F$ is clearly minimized along each such component for $y_v = a_{v,k}$. Let $A\subset S$ denote the set of such $v$ and let $B = S\setminus A$. If $B$ is empty then the proof is complete, so assume it is not. Consider the function $G:\bR^B \ra \bR$ defined by
\[G(y) = F\bigg(y + \sum_{v\in A} a_{v,k} e_v\bigg).\]
Notice that our vectors $x\in X$ from above arise in this fashion. We will determine the minimum of $G$ and this will in turn tell us the minimum of $F$. Let $Y$ denote the subset of $\bR^B$ defined by
\[ a_{v,k} < y_v < a_{v,k+1}\quad\text{for all}\quad v\in B,\]
and $\sum_{v\in B} y_v + \sum_{v\in A} a_{v,k} = 0$. Notice that $Y$ is an open set and that the restriction of our vector $x$ to its $B$ components lies in $Y$ so it is nonempty. Notice further that $G$ is a convex function of $\bR^B$ which is constant on the open set $Y$, therefore, $G$ is minimal on $Y$, as any convex function which is constant on an open set attains its minimum on that set. This now implies that $F$ is minimal on $X$, which completes the proof for all $1\leq k<[L:K]$.

For the remaining cases where $k=0$ or $k=[L:K]$ we make some trivial modifications to our set $X$. For the case $k=0$, we let $X\subset \overline{\varphi(V_{K,S})}\subset \bR^{[L:K]\times S}$ be given by
\[ x_v < a_{v,1}\quad\text{for all}\quad v\in S
\]
and
\[
 \sum_{v\in S} x_v = 0,
\]
where we again use $x_v$ to denote the common value of $x_{v,i}$. Now we demonstrate that $X$ is nonempty by constructing
\[
x_v = a_{v,1} - \frac{s_1}{\# S}.
\]
where $s_i = \sum_{v\in S} a_{v,i}$. In the case $k=[L:K]$ likewise we take $\sum_{v\in S} x_v = 0$ and 
\[
x_v > a_{v,[L:K]}\quad\text{for all}\quad v\in S,
\]
to define our set $X$ and observe that we have a point given by
\[x_v = a_{v,[L:K]} - \frac{s_{[L:K]}}{\# S}\]
(noting that $s_{[L:K]}\leq 0$ in this case). The remainder of the proof continues exactly as above.
\end{proof}

\begin{proof}[Proof of Theorem \ref{thm:maza-fried-mod-units}]
 This is in essence an application of the above theorem with $W$ substituted as the subspace; the primary difference is that we wish to show that in this instance, the infimum claimed is in fact attained in $W$, rather than $\overline{W}$. Suppose without loss of generality that $d_v \log \|\bal\|_v<0$ so that $a_v<0$ (for otherwise we may replace $\bal$ by $\bal^{-1}$ and the height is unaffected). Then \[s=\sum_{w|\infty} d_w \log \|\bal\|_w = \sum_{w|\infty} a_w >0.\] Let $X\subset \overline{\varphi(W)}\subset\bR^S$ (where $S=\{w\in M_K : w|\infty\}\cup\{v\}$) be the set of $x$ satisfying
\[x_v=0,\quad x_w < a_w,\quad\text{for all}\quad w|\infty,\]
and
\[\sum_{w|\infty} x_w = 0.\]
The set $X$ is nonempty as it contains
\[x_w = a_w - s/n\quad\text{for all}\quad w|\infty,\]
where $n = \#\{w\in M_K : w|\infty\}$. But then
\[
\|a-x\|_1 = |a_v| + \sum_{w|\infty} |a_w-x_w| = |a_v| + \sum_{w|\infty} (a_w-x_w)
 = |a_v| + \bigg| \sum_{w|\infty} a_w\bigg|,
\]
and the claim will follow if we can show that this value is minimal, since $[K:\bQ] h_1(\gamma) = \|\varphi(\gamma)\|_1$ for $\gamma\in V_K$. But $X$ is a nonempty open subset of $\overline{\varphi(W)}$ where the convex function $F:\varphi(W)\ra\bR$ given by $y\mapsto \|a-y\|_1$ is constant, therefore, it is the minimum of this function. Since we have an open subset of $\overline{\varphi(W)}$ clearly we have a $\bbeta\in W$ such that $y =\varphi(\bbeta)\in X$ and the proof is complete.
\end{proof}

\begin{proof}[Proof of Theorem \ref{thm:minimizing-mod-wk}]
 By Theorem \ref{thm:orig-maza-fried}, we have a set $X\subset\bR^S$ such that for $\ovrln{\eta}\in \varphi^{-1}(X)\subset \overline{V_{K,S}}$, we have the first condition that $h_1(\bal\ovrln{\eta}^{-1})$ is minimized. Our goal will be to show that if we choose $\ovrln{\eta}\in\varphi^{-1}(X)$ of minimal height, then the remaining two conditions will be satisfied. Let us determine then what the minimal height of $x=\varphi(\ovrln{\eta})\in\bR^S$ can be. In the notation of the proof of Theorem \ref{thm:orig-maza-fried}, we will assume for the moment that $1\leq k<[L:K]$ and write
\[x_v = a_{v,k} + \ep_v + \ep_v'.\]
where $\ep_v,\ep_v'\geq 0$ and
\[
  |x_v| = |a_{v,k}| - \ep_v + \ep_v'.
\]
Clearly we will have $\ep_v=0$ if $a_{v,k}\geq 0$, and likewise $\ep_v' = 0$ if $a_{v,k} + \ep_v < 0$. For a real number $t$ we will denote $t^+ = \max\{t,0\}$ and $t^- = \max\{-t,0\}$, so that $t = t^+ - t^-$ and $|t| = t^+ + t^-$. To minimize $\|x\|_1$ we want to let $\sum_v \ep_v$ be as large as possible, and it is easy to see that we must have $0 \leq \ep_v \leq \min\{a_{v,k}^-, a_{v,k+1} - a_{v,k}\}$. Let
\[ C = \sum_{v\in S} \min\{a_{v,k}^-, a_{v,k+1} - a_{v,k}\}.\]
Our proof will break into two cases. First, assume that $C\geq -\sum_v a_{v,k}$. Then clearly we can choose appropriate $\ep_v$ such that $x_v = a_{v,k} + \ep_v\in [a_{v,k}, a_{v,k+1}]$ for each $v$ and $\sum_v x_v = 0$, or equivalently, that $\sum_v \ep_v = -\sum_v a_{v,k}$, so that
\[
 \|x\|_1 = \sum_{v} |a_{v,k}| - \sum_v \ep_v = \sum_v 2a_{v,k}^+,
\]
and $\|x\|_1$ is clearly minimal. Now for the second case, assume that $C\leq -\sum_v a_{v,k}$. Again, in order to minimize $\|x\|_1$ we want to let $\sum_v \ep_v$ be as large as possible; this is by construction $C$. But we require
\[\sum_v (\ep_v+ \ep_v') = -\sum_v a_{v,k} (\geq 0)\]
in order to have $\sum_v x_v = 0$, so this implies that we will need $\ep_v'$, precisely such that
\[ \sum_v \ep_v' = -\sum_v a_{v,k} -\sum_v \ep_v = -\sum_v a_{v,k} - C. \]
Then clearly
\[\|x\|_1 = \sum_{v} |a_{v,k}| - \sum_v \ep_v + \sum_v \ep_v' = \sum_{v} |a_{v,k}| - \sum_v a_{v,k} - 2C\]
is the minimal height. To prove the third property we need to simplify this expression, so we will now evaluate $C$.
Suppose $\min\{a_{v,k}^-, a_{v,k+1} - a_{v,k}\} = a_{v,k}^-$. Then $a_{v,k+1}\geq 0$, and
\[\min\{a_{v,k}^-, a_{v,k+1} - a_{v,k}\}=a_{v,k}^-=a_{v,k}^- - a_{v,k+1}^-.\]
Now, suppose $\min\{a_{v,k}^-, a_{v,k+1} - a_{v,k}\} = a_{v,k+1} - a_{v,k}$. Then we must have $a_{v,k}\leq a_{v,k+1}\leq 0$, and so
\[\min\{a_{v,k}^-, a_{v,k+1} - a_{v,k}\}=a_{v,k+1} - a_{v,k} = a_{v,k}^- - a_{v,k+1}^-.\]
Thus in general $\min\{a_{v,k}^-, a_{v,k+1} - a_{v,k}\}=a_{v,k}^- - a_{v,k+1}^-$. So now let us continue evaluating $\|x\|_1$ in the case $C\leq -\sum_v a_{v,k}$:
\[\|x\|_1 = \sum_{v} |a_{v,k}| - \sum_v a_{v,k} - 2C 
=  \sum_{v}2a_{v,k}^- - 2\sum_v (a_{v,k}^- - a_{v,k+1}^-) 
 = \sum_{v}2a_{v,k+1}^-.\]
We remark in passing that the condition $C\geq -\sum_v a_{v,k}$ is equivalent to $\sum_v a_{v,k}^+\geq \sum_{v}a_{v,k+1}^-$, so in fact we can express the minimal height of $x$ in both cases as
\[\|x\|_1 = \max\bigg\{\sum_{v} 2a_{v,k}^+ , \sum_{v}2a_{v,k+1}^-\bigg\}.\]

Using such a minimal $\ovrln{\eta} = \varphi^{-1}(x)\in\overline{V_{K,S}}$ we see that the first two claims are satisfied. It remains to show that the third claim is true, specifically, that
\[
 h_1(\ovrln{\eta}) + [L:K]h_1(\bal\ovrln{\eta}^{-1})\leq [L:K] h_1(\bal). 
\]
Translated into the appropriate $L^1$-norms, this claim is equivalent to:
\[\|x\|_{L^1(\bR^S)} + \|a-x\|_{L^1(\bR^{[L:K]\times S})} \leq \|a\|_{L^1(\bR^{[L:K]\times S})}.\]
Where in the term $\|a-x\|_{L^1(\bR^{[L:K]\times S})}$ we view $x$ as a vector in $\bR^{[L:K]\times S}$ via $x_{v,i} = x_v$ for all $i$. Writing this expression out, we have
\[
 \sum_v |x_v| + \sum_{i=1}^{[L:K]}\bigg| \sum_v a_{v,i}\bigg| \leq \sum_{i=1}^{[L:K]} \sum_v |a_{v,i}|,
\]
equivalently, rearranging these terms,
\begin{equation}\label{eqn:xv-ineq}
 2\max\bigg\{\sum_{v} a_{v,k}^+ , \sum_{v}a_{v,k+1}^-\bigg\}\leq 2\sum_v \bigg( \sum_{i=1}^{k} a_{v,i}^+ 
+ \sum_{i=k+1}^{[L:K]} a_{v,i}^- \bigg),
\end{equation}
which is clearly true and completes the proof for the cases $1\leq k<[L:K]$. For the remaining cases, observe that for $k=0$ we have $x_v<a_{v,1}$ and thus it is easy to see that our minimal height is
\[
\sum_v |x_v| = \sum_v 2a_{v,1}^-
\]
and since the right hand side of \eqref{eqn:xv-ineq} holds for $k=0$, the inequality still holds. The $k=[L:K]$ case is similar, as $a_{v,[L:K]} < x_v$ implies our minimal height is
\[
 \sum_v |x_v| = \sum_v 2a_{v,[L:K]}^+.\qedhere
\]
\end{proof}

\subsection{$S$-unit projections and proof of Theorem \ref{thm:extremal-infimum}}
Let $K$ be a finite Galois extension of $\bQ$. We denote the set of places of $K$ by $M_K$. We normalize our absolute values by letting $\|\cdot\|_v$ be the absolute value which extends $|\cdot|_p$ for the rational prime $p$ such that $v|p$, and let $|\cdot|_v=\|\cdot\|_v^{[K_v:\bQ_v]/[K:\bQ]}$. Denote by $S$ a finite set of places to be fixed later which includes all of the archimedean places. Let $O_K$ be the ring of algebraic integers of $K$ and let $U_S$ be the group of $S$-units of $K$. Since $S$ is finite and contains the archimedean places, we know by Dirichlet's $S$-unit theorem that $U_S$ is a free abelian group of finite rank $s=\# S-1$. Recall that the class group is the group of nonzero fractional ideals of $K$ modulo principal ideals. It is well-known that for number fields, the class group of a number field has a finite order, and we will denote the order of the class group of $K$ by $h$. It follows immediately that if for some finite place $v\in M_K$ the ideal
\[
 \cP_v = \{\al\in K : \|\al\|_v<1\}\subset O_K
\]
is not principal, then
\begin{equation}\label{eqn:power-is-principal}
 \cP_v^h = (\al)\subset O_K
\end{equation}
is a principal ideal of $O_K$, since the class of $\cP_v^h$ is trivial in the class group.

The goal of this section is to construct a projection $P_S : V_K\ra V_{K,S}$ which will be instrumental in the proof of the main theorem. Let $S$ consist of the following places of $K$:
\begin{enumerate}
 \item The archimedean places of $K$.
 \item The support of $\bal$ (all places where $\bal$ has nontrivial valuation).
 \item The Galois conjugates of the above places under the natural action $\|\cdot \|_{\sigma v} = \|\sigma^{-1}(\cdot)\|_v$.
\end{enumerate}
It is clear that $S$ is finite. We now proceed to associate a generator to each place outside of $S$:
\begin{lemma}\label{lemma:alpha-v-exists}
 For any $v\in M_K\setminus S$, we can find $\bal_v\in V_K$ such that
\begin{enumerate}
 \item $\|\bal_v\|_v<1$,
\item $\|\bal_v\|_w = 1$ for all $w\in M_K\setminus S$ with $w\neq v$, and 
\item $\|\bal_v\|_w \geq 1$ for all $w\in S$.
\item $h_1(\bal_v) = \inf_{\bbeta \in V_{K,S}} h_1(\bal_v/\bbeta)$.
\end{enumerate}
\end{lemma}
\begin{proof}
If $\cP_v=\{\al\in K : \|\al\|_v<1\}\subset O_K$ is a principal ideal, then let $\beta$ be a generator. Otherwise, let $\cP_v^h = (\al)$ as in \eqref{eqn:power-is-principal} and let $\bbeta = \bal^{1/h}\in V_K$. Clearly, $\bbeta$ has a nontrivial finite valuation only at $v$ of $\|\bbeta\|_v=p^{-1/e}$, where $e$ is the ramification index of $v|p$. By Theorem \ref{thm:maza-fried-mod-units} above, we can find $\ovrln{\eta}\in V_{K,S}$ such that
\begin{align*}
 h_1(\bbeta\ovrln{\eta}) &= \sum_{w\in M_K} |\log |\beta\eta|_w| \\
&= \sum_{w\in M_K\setminus S} |\log |\beta|_w| + \sum_{w\in S} |\log |\beta\eta|_w|\\
&= |\log |\beta|_v| + \bigg|\sum_{w\in S} \log |\beta\eta|_w\bigg|\\
&= |\log |\beta|_v| + \bigg|\sum_{w\in S} \log |\beta|_w\bigg|.
\end{align*}
That we have equality above implies that either $\log |\beta\eta|_w\geq 0$ for all $w\in S$ or $\log |\beta\eta|_w\leq 0$ for all $w\in S$. By our choice of  $\beta$ we have $\log |\beta|_v<0$, and hence, by the product formula, all of the $S$ valuations of $\beta\eta$ must be nonnegative. We therefore can choose $\bal_v=\bbeta\ovrln{\eta}$ and we are done.
\end{proof}

Let $v\in M_K$ and suppose $v|p$ for the rational prime $p$. Let $G=\Gal(\Qbar/\bQ)$ be the absolute Galois group, and let
\[
 H = \Stab_G(v)
\]
be the decomposition group associated to the finite place $v$. Let $\bal\in V_K$ and take $\{\sigma_1, \ldots, \sigma_k\}$ to be a set of right coset representatives for $\Stab_H(\bal)$ in $H$ (where $k=[H:\Stab_H(\bal)]$) and then define $P_H : V_K\ra V_K$ to be
\[P_H \bal = \left(\sigma_1(\bal)\cdots \sigma_k(\bal)\right)^{1/k}.\]
Then by Proposition \ref{prop:PK-continuous-wrt-Weil-norm}, $P_H$ is a projection to $V_F\subseteq V_K$, for $F\subseteq K$ the fixed field of $H$, of operator norm one with respect to the Weil $p$-height $h_p$ for $1\leq p\leq\infty$.
We will now construct a system of $\bal_v$ for each place $v\in M_K\setminus S$.
\begin{lemma}\label{lemma:alpha-v-system}
There exists a set $\{\bal_v\in V_K : v\in M_K\setminus S\}$ such that each $\bal_v$ satisfies the conditions of Lemma \ref{lemma:alpha-v-exists} above with the following additional property: for any $w\in S$ and $\sigma\in G$, if $\|\al_v\|_w \neq \|\al_v\|_{\sigma w}$ then $\sigma v\neq v$.
\end{lemma}
\begin{proof}
For each rational prime $p$ which has a place in $M_K\setminus S$ lying above it, pick one particular place $v|p$ lying above it. Choose $\bal_v'$ to be the number constructed by Lemma \ref{lemma:alpha-v-exists} above, and let $\bal_v = P_H \bal_v'$ where $H=\Stab_G(v)$ is the stabilizer of the place $v$ in the absolute Galois group as above. Notice that by the fact that $P_H$ has norm one and by minimality modulo $V_{K,S}$ of $\bal_v'$ in Lemma \ref{lemma:alpha-v-exists}, $h_1(P_H \bal_v') = h_1(\bal_v')$. Since $S$ is closed under the Galois action, and $H$ fixes the place $v$, $\bal_v$ still satisfies the criteria of Lemma \ref{lemma:alpha-v-exists}. For any other place $w|p$ lying above the same rational prime $p$, observe that there exists $\sigma\in G$ with $\sigma v = w$. Define $\bal_w = \sigma^{-1}(\bal_v)$, and repeat this construction for every rational prime $p$ whose extensions to $K$ lie in $M_K\setminus S$. This gives us the entire set of $\bal_v$ whose existence we need to establish, and since the Galois action permutes the places $v$ lying over $p$, the $\bal_v$ thus constructed all meet the conditions of Lemma \ref{lemma:alpha-v-exists}. 

It now remains to see that this set satisfies the additional property claimed. This is guaranteed by the ``averaging'' over $H$ done by $P_H$ in constructing the original $\bal_v$ whose orbit we took in the above construction. Observe that if $\sigma\in G$ fixes the $v$-adic valuation of $\bal_v$, then $\sigma\in H$.  Let $F\subseteq K$ denote the fixed field of $H$ and view $P_H$ as the projection to $V_F$.  Then $\bal_v\in V_F$ is some power of an element of $F^\times/\Tor(F^\times)$, so by linearity, we have  $\sigma\bal_v=\bal_v$.
Thus we see that for such $\sigma\in G$, $\|\sigma \bal_v\|_w = \|\bal_v\|_{\sigma^{-1} w}$ unless $\sigma v\neq v$, in which case we have the desired conclusion.
\end{proof}

\begin{cor}
 For $v|p$ and $\bal_v$ in the set as constructed in Lemma \ref{lemma:alpha-v-system}, $\delta(\bal_v)$ is precisely the number of places of $K$ which lie over $p$.
\end{cor}
\begin{proof}As seen in the proof, if $\sigma(\bal_v)\neq\bal_v$, then $\sigma v\neq v.$ While, if $\sigma(\bal_v)=\bal_v$, then $1>\|\bal_v\|_v=\|\sigma(\bal_v)\|_v=\|\bal_v\|_{\sigma^{-1}v}$, which gives $\sigma v=v.$ 
\end{proof}
We are now ready to construct the projection $P_S : V_K\ra V_{K,S}$ which is fundamental to the proof of Theorem \ref{thm:extremal-infimum}.

\begin{prop}\label{prop:s-unit-projection}
 There exists a projection $P_S : V_K\ra V_{K,S}$ which satisfies the following properties:
\begin{enumerate}
 \item $h_1(P_S \bal) \leq h_1(\bal)$, so $\|P_S\|=1$ with respect to the Weil height norm, and
 \item $\delta(P_S\bal) \leq \delta(\bal)$, and thus $\|P_S\|=1$ with respect to the Mahler norm.
\end{enumerate}
\end{prop}
\begin{proof}
 For our given $S$, let $\{\bal_v : v\in M_K\setminus S\}$ be the set constructed by Lemma \ref{lemma:alpha-v-system}. For each $v\in M_K\setminus S$, define the map $n_v : V_K\ra \bQ$ via the requirement that
\[
 \left\|\bbeta\bal_v^{-n_v(\bbeta)}\right\|_v = 1\quad\text{for all}\quad \bbeta\in V_K.
\]
It is easy to see that such a value for $n_v$ must exist and be unique, since the $v$-adic valuations are discrete.\footnote{The reader will note that by our choice of $\bal_v$, the function $n_v(\cdot)$ is essentially the linear extension of $\ord_v(\cdot)$ from $K^\times/\Tor(K^\times)$ to $V_K$.} Further, observe that
\[
 n_v(\bbeta\ovrln{\gamma}) = n_v(\bbeta) + n_v(\ovrln{\gamma})\quad\text{for all}\quad \bbeta,\ovrln{\gamma}\in V_K.
\]
Define the map
\begin{align*}
 P_S : V_K &\ra V_{K,S} \\
     \bal &\mapsto \bal \prod_{v\in M_K\setminus S} {\bal_v}^{-n_v(\bal)}
\end{align*}
That this is well-defined follows from the fact that $n_v(\bal)=0$ for all but finitely many $v$ and from the fact that by our choice of $\bal_v$ and $n_v(\bal)$, $P_S \bal$ has support only in $S$ and thus belongs to the $\bQ$-vector space span of the $S$-units $V_{K,S}$. 

We will now prove that $P_S$ satisfies the first desired property. Fix our $\bal\in V_K$ and let $\bbeta=P_S\bal\in V_{K,S}$. Let $T$ denote the Galois orbit of $\supp(\bal)\setminus S$ inside $M_K$. The claim is then that
\[
 h_1(\bbeta)\leq h_1\bigg(\bbeta\prod_{v\in T} \bal_v^{n_v}\bigg)=h_1(\bal),
\]
 where we will suppress the argument in the exponents $n_v=n_v(\bal)$. Denote $S' = M_K\setminus S$. Then
 \begin{equation}\label{eqn:h-P-alpha}
 h_1(\bbeta) = \sum_{w\in S} \left| \log |\bbeta|_w \right| + \sum_{w\in S'} \left| \log |\bbeta|_w \right|.
\end{equation}
Now, $\sum_{w\in S'} \left| \log |\bbeta|_w \right| = 0$, since $\bbeta\in V_{K,S}$. We apply the triangle inequality to the remaining term:
\begin{equation}\label{eqn:S-term}
\sum_{w\in S} \left| \log |\bbeta|_w \right|
\leq \sum_{w\in S} \left| \log |\bbeta|_w + \sum_{v\in T}n_v \log |\bal_v|_w \right|
+ \sum_{w\in S} \left| \sum_{v\in T}n_v \log |\bal_v|_w \right|.
\end{equation}
Observe that by our choice of $\bal_v$ in the lemmas above, $|\bal_v|_w\geq 1$ for all $w\in S$, and thus,
\[
\sum_{w\in S} \left| \sum_{v\in T}n_v \log |\bal_v|_w \right| 
\leq \sum_{w\in S} \sum_{v\in T}|n_v| \log |\bal_v|_w
=\sum_{w\in S'} \sum_{v\in T}|n_v| (-\log |\bal_v|_w),
\]
where the last equality follows from the product formula. But likewise, $|\bal_v|_w=1$ for all $w\in S'\setminus \{v\}$ and $|\bal_v|_v<1$, so in fact,
\[
 \sum_{w\in S} \left| \sum_{v\in T}n_v \log |\bal_v|_w \right| 
 \leq \sum_{w\in S'} \left| \sum_{v\in T} n_v \log |\bal_v|_w\right|.
\]
On observing that $|\bbeta|_w=1$ for all $w\in S'$, we may write this same expression as:
\begin{equation}\label{eqn:S-S'-ineq}
 \sum_{w\in S} \left| \sum_{v\in T}n_v \log |\bal_v|_w \right| 
 \leq \sum_{w\in S'} \left| \log |\bbeta|_w + \sum_{v\in T} n_v \log |\bal_v|_w\right|.
\end{equation}
Combining equations \eqref{eqn:h-P-alpha}, \eqref{eqn:S-term}, and \eqref{eqn:S-S'-ineq}, we find that
\[
 h_1(\bbeta) \leq \sum_{w\in M_K} \left| \log |\bbeta|_w + \sum_{v\in T} n_v \log |\bal_v|_w\right|=h_1\bigg(\bbeta \prod_{v\in T} \bal_v^{n_v}\bigg),
\]
which is the desired result.

It now remains to prove the second claim, namely that
\[
 \delta(\bbeta)\leq \delta\bigg(\bbeta \prod_{v\in T} \bal_v^{n_v}\bigg)=\delta(\bal).
\]
Suppose for some $\sigma\in G$ that $\bbeta\neq\sigma\bbeta$ but $\sigma(\bal)=\bal$.  Then for some $w\in S$, $\|\bbeta\|_w\neq\|\bbeta\|_{\sigma w}$, and so we must have
\[
 \bigg\|\prod_{v\in T} \bal_v^{n_v}\bigg\|_w\neq \bigg\|\prod_{v\in T} \bal_v^{n_v}\bigg\|_{\sigma w}.
\]
It follows then by Lemma \ref{lemma:alpha-v-system} that for some $v\in T$ we must have $\sigma v\neq v$ and $n_v\neq n_{\sigma v}$, else the $w$-adic valuation would not differ. But then it is easy to see that
\[
 \bigg\|\prod_{u\in T} \bal_{u}^{n_u}\bigg\|_v=\|\bal_v\|_v^{n_v} =p^{-n_v/e}\neq p^{-n_{\sigma v}/e}= \|\bal_{\sigma v}\|_{\sigma v}^{n_{\sigma v}} = \bigg\|\prod_{u\in T} \bal_u^{n_u}\bigg\|_{\sigma v},
\]
where $e$ is the ramification index of $v|p$.  Thus any contribution the $\prod_{u\in T} \bal_u^{n_u}$ term might have towards decreasing the orbit of $\bal=\bbeta\prod_{v\in T} \bal_v^{n_v}$ by equating two $w$-adic valuations of $\bal$ for $w\in S$ will nevertheless result in distinct $v$-adic valuations for some $v\in T$ and thus the new orbit will be at least as large, proving the claim.
\end{proof}

We are now ready to prove Theorem \ref{thm:extremal-infimum}.

\begin{proof}[Proof of Theorem \ref{thm:extremal-infimum}]
Let $\bal\in V_K$, where $K$ is the Galois closure of the minimal field of $\bal$. Let $P_K:\cG\ra V_K$ be the projection to $V_K$, $S$ the set constructed above for $K$ so that in fact $\bal\in V_{K,S}$, and $P_S : V_K\ra V_{K,S}$ the projection defined in Proposition \ref{prop:s-unit-projection}, where $V_{K,S}$ is the $\bQ$-vector space span of the $S$-units in $K^\times$ modulo torsion. Notice that in fact, for some set $S'\subset M_{\bQ}$, we have
\[
 \bigcup_{v\in S} \{w\in M_{\Qbar}:w|v\in M_K\} = \bigcup_{p\in S'} \{w\in M_{\Qbar}:w|p\in M_\bQ\}
\]
by the requirement that $S$ be closed under the Galois action. $S$, as a set of places on $K$, meets the criteria set forth in the theorem statement. Let $P=P_SP_K:\cG\ra V_{K,S}$.  By Lemma \ref{lemma:delta-PK-leq-for-K-in-Kg} and Propositions \ref{prop:PK-continuous-wrt-Weil-norm} and \ref{prop:s-unit-projection},  we have that $\delta h_1(P\bbeta)\leq \delta h_1(\bbeta)$ for all $\bbeta\in\cG$. Since $P$ is linear and $\bal\in V_{K,S}$, note that $\bal = P \bal$, so if we have a factorization of $\bal$ into $\bal_i\in \cG$ for $i=1,\ldots,n$, then
\[\bal = \bal_1\cdots \bal_n \implies \bal = (P\bal_1) \cdots (P\bal_n),\]
and $P\bal_i \in V_{K,S}$ for all $i=1,\ldots,n$. Then by our established inequalities for $P_K$ and$P_S$ with respect to $\delta h_1$,
\[
\sum_{i=1}^n \delta h_1 (P \bal_i) \leq \sum_{i=1}^n \delta h_1(\bal_i).
\]
Hence we may take the infimum within $V_{K,S}$.  Associate to each term in the infimum its minimal subspace $V_{F,S}\subseteq V_{K,S}$ containing it for $F\subset K$.  If we have more than one term for any given minimal subspace $V_{F,S}$, notice that the $\delta$ values are equal and we can combine any such terms by the triangle inequality for $h_1$. Thus, the first part of the claim is proven. The remaining criterion easily follows from observing that the choice of $\bal_F$ can be made in accord with Theorem \ref{thm:minimizing-mod-wk}.
\end{proof}

\end{document}